\documentclass{amsart}

\usepackage{amsfonts}
\usepackage{latexsym}
\usepackage{amssymb}
\usepackage{amsmath}
\usepackage{color}

\newcommand{\R}{\mathbb R}

\newcommand{\K}{\mathcal K}

\newtheorem{thm}{Theorem}[section]
\newtheorem{cor}{Corollary}[section]
\newtheorem{lemma}{Lemma}[section]
\newtheorem{df}{Definition}[section]
\newtheorem{proposition}{Proposition}[section]
\theoremstyle{remark}
\newtheorem*{rmk}{Remark}

\numberwithin{equation}{section}

\begin{document}

\title{volume inequalities for the $i$-th-convolution bodies}

\author{David Alonso-Guti\'{e}rrez}
\email{alonsod@uji.es}
\address{Departament de Matem\`atiques, Universitat Jaume I, Campus de Riu Sec, E-12071, Castell\'o de la Plana, Spain.}

\author{Bernardo Gonz\'alez}
\email{bgmerino@um.es}
\address{Departamento de Matem\'aticas, Universidad de Murcia, Campus Espinardo, 30100 Murcia, Spain}

\author{Carlos Hugo Jim\'enez}
\email{carloshugo@us.es}
\address{ Departamento de An\'alisis Matem\'atico, Apartado de Correos 1160, Sevilla, 41080, Spain}


\date{\today}
\begin{abstract}
 We obtain a new extension of Rogers-Sephard inequality providing an upper bound for the volume of the sum of two convex bodies $K$ and $L$. We also give lower bounds for the volume of the $k$-th limiting convolution body of two convex bodies $K$ and $L$. Special attention is paid to the $(n-1)$-th limiting convolution body, for which a sharp inequality, which is equality only when $K=-L$ is a simplex, is given. Since the $n$-th limiting convolution body of $K$ and $-K$ is the polar projection body of $K$, these inequalities can be viewed as an extension of Zhang's inequality.
\end{abstract}
\maketitle

\section{Introduction and notation}

Given $K\in\K^n_0$ an $n$-dimensional convex body ({\it i.e.} convex, compact subset of $\R^n$ with non-empty interior) and $\theta\in S^{n-1}$ a vector in the unit Euclidean sphere, we denote by $P_{\theta^\perp}(K)$ the projection of $K$ onto the hyperplane orthogonal to $\theta$. An important object in the study of hyperplane projections of a convex body is its polar projection body, since it gathers the information about the volume of all of its hyperplane projections. Namely, the polar projection body of $K$, which is denoted by $\Pi^*(K)$, is the centrally symmetric convex body which is the unit ball of the norm
$$
\Vert x\Vert_{\Pi^*(K)}=|x||P_{x^\perp}(K)|,
$$
where by $|\cdot|$ we denote, when no confusion is possible, indistincly the usual Lebesgue measure of a set and the Euclidean norm of a vector.

For any $T\in GL(n)$ we have that $\Pi^*(TK)=|\det T|^{-1}T\Pi^*(K)$ and then the quantity $|K|^{n-1}|\Pi^*(K)|$ is affine invariant. Perhaps the most important inequalities involving the polar projection body are Petty's projection \cite{P} and Zhang's inequality \cite{Z}.  On one hand, Petty's projection inequality states that the afforementioned affine invariant quantity is maximized when $K$ is an ellipsoid. Thus, denoting by $B_2^n$ the $n$-dimensional Euclidean ball,
\begin{equation}\label{PettyProjectionInequality}
|K|^{n-1}|\Pi^*(K)|\leq|B_2^n|^{n-1}|\Pi^*(B_2^n)|=\left(\frac{|B_2^n|}{|B_2^{n-1}|}\right)^n.
\end{equation}

On the other hand, Zhang proved a reverse form of (\ref{PettyProjectionInequality}), showing that this quantity is minimized when $K$ is a simplex. Thus, denoting by $\Delta^n$ the $n$-dimensional regular simplex,
\begin{equation}\label{ZhangInequality}
|K|^{n-1}|\Pi^*(K)|\geq|\Delta_n|^{n-1}|\Pi^*(\Delta_n)|=\frac{1}{n^n}\left(\begin{array}{c}2n\\n\end{array}\right).
\end{equation}

For any $K\in\K^n_0$, Steiner's formula says that the volume of $K+tB_2^n$ (where the sum is the Minkowski addition of two sets) can be expressed as a polynomial in $t$
$$
|K+tB_2^n|=\sum_{k=0}^n\binom{n}{k}
W_k(K)t^k.
$$
The coefficients $W_k(K)$ are called the querma\ss integrals of $K$ and, by Kubota's formula, they can be expressed
$$
W_{n-k}(K)=\frac{|B_2^{n}|}{|B_2^k|}\int_{G_{n,k}}|P_E(K)|d\nu_{n,k}(E),
$$
where $G_{n,k}$ denotes the Grassmannian manifold of the linear $k$-dimensional subspaces of $\R^n$, $d\nu_{n,k}$ is the unique Haar probability measure, invariant under orthogonal maps, on $G_{n,k}$ and $P_E$ denotes the orthogonal projection onto the subspace $E$. Notice that $W_0(K)=|K|$, $nW_1(K)=|\partial K|$ (the surface area of $K$) and $W_{n-1}(K)=|B_2^n|w(K)$, (the mean width of $K$). We refer the reader to \cite{SCH} for these and many other well-known facts in the Brunn-Minkowski theory.

In the same way as the volume of the $(n-1)$-dimensional projections of $K$ define a norm in $\R^n$, the querma\ss integrals of the $(n-1)$-dimensional projections also define a norm, whose unit ball is the $i$-th polar projection body. Namely, if $1\leq i\leq n-1$, $\Pi_i^*(K)$ is the unit ball of the norm given by
$$
\Vert x\Vert_{\Pi_i^*(K)}=|x|W_{n-i-1}(P_{x^{\perp}}(K))=\frac{1}{2}\int_{S^{n-1}}|\langle u,x\rangle|dS_i(K,u),
$$
where $dS_i(K,u)$ denotes the $i$-th surface area measure of $K$. Notice that the $(n-1)$-th polar projection body is exactly the polar projection body defined before, $\Pi^*(K)=\Pi_{n-1}^*(K)$. However, when $i\neq n-1$, it is no longer true that $|K|^{i}|\Pi_i^*(K)|$ is an affine invariant.

In \cite{L1}, \cite{L2} and \cite{L3}, the author studied the class of mixed projection bodies and gave sharp inequalities for them and their polars. Since the $i$-th polar projection bodies belong to this class, the following inequality which extends (\ref{PettyProjectionInequality}) was obtained:
\begin{equation}\label{iPetty}
|K|^{i}|\Pi_{i}^*(K)|\leq|B_2^n|^{i}|\Pi_{i}^*(B_2^n)|=\frac{|B_2^n|^{i+1}}{|B_2^{n-1}|^n},
\end{equation}
with equality if and only if $K=B_2^n$.

This inequality was strengthened in \cite{L3}. When $i=n-1$, Zhang's inequality gives a lower bound for the quantity $|K|^{i}|\Pi_i^*(K)|$. From the results in \cite{L3}, one can easily deduce (see Section \ref{LutwakLowerBound}) the following lower bound for any $i$
\begin{equation}\label{iZhang}
|K|^i|\Pi_{i}^*(K)|\geq\frac{1}{n^n}{2n\choose n}\frac{|K|^{i+1}}{W_{n-i-1}(K)^n}.
\end{equation}
However, there are no equality cases in this inequality unless $i=n-1$.

In \cite{AJV}, the authors studied the behavior of the $\theta$-convolution body of two convex bodies
$$
K+_{\theta}L=\{x\in K+L\,:\,|K\cap(x-L)|\geq\theta M(K,L)\},
$$
where $\displaystyle{M(K,L)=\max_{z\in\R^n}|K\cap(z-L)|}$. In particular, since $$\lim_{\theta\to 1^-}\frac{K+_\theta(-K)}{1-\theta^\frac{1}{n}}=n|K|\Pi^*(K)$$ (see \cite{S}), a new proof of Zhang's inequality (\ref{ZhangInequality}) was obtained and this inequality was extended to the limiting convolution body of two different convex bodies:
$$
\left|\lim_{\theta\to 1^-}\frac{K+_\theta L}{1-\theta^\frac{1}{n}}\right|\geq{2n\choose n}\frac{|K||L|}{M(K,L)}.
$$
The results in this paper also characterized the equality cases in Rogers-Sephard inequality \cite{RS}:
\begin{equation}\label{RogersSephard}
M(K,L)|K+L|\leq{2n\choose n}|K||L|.
\end{equation}

In \cite{TS}, the author considered a different class of convolution bodies of two convex bodies ($k$-th $\theta$-convolution bodies) and studied their limiting behavior when $\theta$ tends to 1. Changing slightly the definition in \cite{TS}, the $k$-th $\theta$-convolution body of $K$ and $L$ is:
$$
K+_{k,\theta}L:=\{x\in K+L\,:\,W_{n-k}(K\cap(x-L))\geq\theta M_{n-k}(K,L)\},
$$
where $\displaystyle{M_{n-k}(K,L)=\max_{x\in K+L}W_{n-k}(K\cap(x-L))}$. Notice that $K+_{n,\theta} L=K+_{\theta}L$.

In this paper we are going to follow the lines of \cite{AJV} and study some properties of this class of convolution bodies, all this in order to prove some volume inequalities for the limiting convolution body and $K+L$ that can be viewed as an extension of Zhang's inequality and Rogers-Sephard inequality for the volume of the difference body.

We give an upper bound for the volume of the sum of $K$ and $L$ and a lower bound for the volume of the limiting $k$-th convolution body of $K$ and $L$
$$
C_k(K,L):=\lim_{\theta\to 1^-}\frac{K+_{k,\theta}L}{1-\theta^\frac{1}{k}}.
$$
Special attention is paid to the case $k=n-1$, for which the inequalities we obtain are sharp and improve inequality (\ref{iZhang}):
\begin{thm}\label{Theorem}
Let $K,L\in\K^n_0$. Then
$$
|C_{n-1}(K,L)|\geq{2n \choose n}\frac{|K|W_{1}(L)+|L|W_{1}(K)}{2M_{1}(K,L)}\geq|K+L|
$$
with equality in each one of the inequalities if and only if $K=-L$ is a simplex.
\end{thm}

The left-hand side inequality improves inequality (\ref{iZhang}), when $L=-K$ and $k=i+1=n-1$ since, as we will see in Section \ref{LutwakLowerBound}, for any $1\leq k\leq n$ and any $K\subseteq\R^n$
\begin{equation}\label{convolutionandprojection}
C_k(K,-K)\subseteq nW_{n-k}(K)\Pi_{k-1}^*(K).
\end{equation}

The right hand-side inequality gives an upper bound for the volume of the sum of two convex bodies $K$ and $L$ of a different nature than Rogers-Shephard inequality. Excluding the case when $L=-K$ is a simplex, for which we know Rogers-Shephard inequality is sharp, the upper bound in Theorem \ref{Theorem} seems to give a better bound for the volume $|K+L|$ than (\ref{RogersSephard}). Indeed, it is easy to see the latter for $K$ and $L=-\lambda K$ with $\lambda>1$.

In \cite{R}, the author gave an upper bound for the volume of the sections of the difference body. Namely, he proved that for any $E\in G_{n,k}$
\begin{equation}\label{SectionsDifferenceBody}
|(K-K)\cap E|\leq C^k\varphi(n,k)^k\max_{x\in\R^n}|K\cap(x+E)|,
\end{equation}
where
$$
\varphi(n,k)=\min\left\{\frac{n}{k},\sqrt k\right\}.
$$

This estimate was used in \cite{R2} to give an upper bound of $M(K)M^*(K)$ for any convex body $K$ and consequently gave an upper bound for the Banach-Mazur distance between any two convex bodies (non-necessarily symmetric). In order to prove the $\frac{n}{k}$ upper bound the author proved some estimates than can be seen as volume inequalities for the $k$-th, $\theta$ convolution bodies of $K$ and $-K$. We will provide some volume estimates for the sections of the sum of two convex bodies that, as a particular case, will recover Rudelson's $\frac{n}{k}$ upper bound providing a simpler proof of it.

The paper is organized as follows: In Section \ref{ConvolutionBodies} we define the class of convolution bodies we will use and study some of their general properties. Since inequality (\ref{iZhang}) is not explicitly written in \cite{L3}, we show how it is deduced from the results there in Section \ref{LutwakLowerBound}. We also prove (\ref{convolutionandprojection}) to show that Theorem \ref{Theorem} is really an improvement of equation (\ref{iZhang}) when $k=i+1=n-1$. In Section \ref{Proof} we give a lower bound for the volume of $C_{k}(K,L)$ which in particular gives the proof of Theorem \ref{Theorem}. Finally in Section
\ref{SectionsProjDiff} we provide bounds for the volume of sections of the limiting convolution body $C_n(K,L)$ and the body $K+L$.

We denote by $\textrm{span}\{x_1,\dots, x_m\}$ the smallest linear subspace that contains the vectors $x_1,\dots,x_m$. The 1-dimensional linear subspace generated by a vector $x$ will be denoted by $\langle x\rangle$. The interior of a set $A$ will be denoted by $\textrm{int}(A)$. If $A$ is contained in an affine subspace, $\textrm{int}(A)$ refers to the relative interior of $A$ in such subspace.

\section{The $h,\theta$-convolution bodies.}\label{ConvolutionBodies}
\begin{df}\label{Definition}
Let $h:\K^n_0\rightarrow\R$ satisfying
\begin{itemize}
\item[$(i)$] If $K\subseteq L$ then $h(K)\leq h(L)$, for any $K,L\in\K^n_0$.
\item [$(ii)$]$h(a+K)=h(K)$, for any $a\in\R^n$ and $K\in\K^n_0$.
\item [$(iii)$]$h(\lambda K)=\lambda^kh(K)$ for any $0\leq \lambda\leq 1$, $K\in\K^n_0$ and some integer  $k$,
\item [$(iv)$] $h$ satisfies a Brunn-Minkowski type inequality $$h((1-\lambda)K+\lambda L)^\frac{1}{k}\geq \lambda h(K)^\frac{1}{k}+\lambda h(L)^\frac{1}{k}.$$
\end{itemize}

We define the $h,\theta$-convolution of $K$ and $L$ by
$$
K+_{h,\theta}L:=\{x\in K+L\ :\ h(K\cap (x-L))\geq \theta M_h(K,L))\},
$$
where $\displaystyle{M_h(K,L)=\max_{z\in K+L} h(K\cap(z-L))}$. For all of our results, we can assume without loss of generality that $M_h(K,L)=K\cap(-L)$.
\end{df}
\begin{rmk}
The querma\ss integrals $W_{n-k}(K)$ satisfy these hypotheses. In that case we have denoted $K+_{W_{n-k},\theta}L= K+_{k,\theta}L$.
\end{rmk}

The following proposition gives an inclusion relation between the $h,\theta$-convolution bodies.
\begin{proposition}\label{incconvexity}
    Let $K,L\in\K^n_0$. Then for every $\theta_1,\theta_2, \lambda_1,\lambda_2 \in[0,1]$ such that $\lambda_1+\lambda_2\leq 1$ we have
   $$
    \lambda_1 (K+_{h,\theta_1}L)+ \lambda_2(K+_{h,\theta_2}L)\subseteq K+_{h,\theta} L,
    $$
    where $1-\theta^{\frac{1}{k}}=\lambda_1(1-\theta_1^{\frac{1}{k}})+\lambda_2(1-\theta_2^{\frac{1}{k}}).$
    \end{proposition}

    \begin{proof}
    Let $x_1\in K+_{h,\theta_1}L$ and $x_2\in K+_{h,\theta_2}L$. From the general inclusion
$$
K\cap(\lambda_0A_0+\lambda_1A_1+\lambda_2A_2)
\supset
\lambda_0K\cap A_0+\lambda_1K\cap A_1+\lambda_2K\cap A_2
$$
where $K$ is convex and $\lambda_0+\lambda_1+\lambda_2=1$, and using the convexity of $K$ and $L$, we have
    $$
    K\cap (\lambda_1 x_1+\lambda_2 x_2-L)\supseteq(1-\lambda_1-\lambda_2)(K\cap(-L))+\lambda_1[K\cap(x_1-L)]+\lambda_2[K\cap (x_2-L)].
    $$
    By the properties of $h$ and the fact that $x_i\in K+_{h,\theta_i}L$ we have
    $$
    h(K\cap(\lambda_1x_1+\lambda_2x_2-L))\geq[1-\lambda_1(1- \theta_1^{\frac{1}{k}})-\lambda_2(1-\theta_2^{\frac{1}{k}})]^kM(K,L),
    $$
    which proves that $\lambda_1x_1+\lambda_2x_2\in K+_{h,\theta} L$ for $\theta=[1-\lambda_1(1- \theta_1^{\frac{1}{k}})-\lambda_2(1-\theta_2^{\frac{1}{k}})]^k$.
     \end{proof}

Taking $\theta_1=\theta_2$ and $\lambda_2=1-\lambda_1$ we have

\begin{cor}
Let $K,L\in\K^n_0$ and $\theta\in[0,1]$. Then $K+_{h,\theta} L$ is  convex.
\end{cor}

\begin{cor}\label{increasingintheta}
Let $K,L\in\K^n_0$. Then, for every $0\leq\theta_0\leq \theta<1$ we have
$$
\frac{K+_{h,\theta_0} L}{1-{\theta_0}^\frac{1}{k}}
\subseteq
\frac{K+_{h,\theta} L}{1-\theta^\frac{1}{k}}.
$$
\end{cor}

\begin{proof}
Taking $\theta_1=\theta_2=\theta_0$ in the above proposition, for any $\lambda_1,\lambda_2\in [0,1]$ such that $\lambda_1+\lambda_2\leq 1$
$$
(\lambda_1+\lambda_2)(K+_{h,\theta_0} L)
=
\lambda_1(K+_{h,\theta_0} L)+
\lambda_2(K+_{h,\theta_0} L)
\subseteq
 K+_{h,\theta} L,
$$
with $1-\theta^{\frac{1}{k}}=(\lambda_1+\lambda_2)(1-{\theta_0}^\frac{1}{k})$. Since $\lambda_1+\lambda_2=\displaystyle{\frac{1-\theta^\frac{1}{k}}{1-\theta_0^\frac{1}{k}}}$,
$$
\frac{1-\theta^\frac{1}{k}}{1-\theta_0^\frac{1}{k}}(K+_{h,\theta_0} L)
\subseteq
K+_{h,\theta} L
$$
whenever $\lambda_1+\lambda_2\leq 1$, which means $0\leq\theta_0\leq \theta\leq 1$.

\end{proof}

The next proposition shows that if the equality cases in ({\it iv}) of Definition \ref{Definition} occur $K$ and $L$ must be homothetic. Thus,
it is a necessary condition for $K=-L$ to be a simplex in order to attain equality in all inequalities in Corollary \ref{increasingintheta}. This is the case if $h(K)=W_{n-k}(K)$ ($k>n-1$).

\begin{lemma}\label{simplex}
Let $h$ be like in Definition \ref{Definition}, such that equality in ({\it iv}) occurs if and only if $K$ and $L$ are homothetic. Assume that for every $0\leq\theta_0\leq \theta<1$ we have
$$
\frac{K+_{h,\theta_0} L}{1-{\theta_0}^\frac{1}{k}}
=
\frac{K+_{h,\theta} L}{1-\theta^\frac{1}{k}}.
$$
Then $K=-L$ is a simplex.
\end{lemma}

\begin{proof}
In particular, we have that for any $0\leq\theta<1$
$$
K+_{h,\theta} L
=
(1-{\theta}^\frac{1}{k})(K+L)
$$
and
$$
K+_{h,1}L=\{0\}.
$$
Thus, for any $x\in K+L$, $x\in \partial (K+_{h,\theta} L)$ for some $\theta$ and
$$
x=\theta^{\frac{1}{k}}0+(1-\theta^\frac{1}{k})y,
$$
with $y\in K+L$. Since $x\in\partial (K+_{h,\theta} L)$ we have $h(K\cap(x-L))=\theta M_{h}(K,L)$ and so, we have equality in
$$
h^\frac{1}{k}(K\cap(x-L))\geq h^\frac{1}{k}(\theta^\frac{1}{k}(K\cap (-L))+((1-\theta^\frac{1}{k})(K\cap(y-L)))\geq \theta^\frac{1}{k}M(K,L)^\frac{1}{k}.
$$
Thus, $K\cap(x-L)$, $K\cap(-L)$ and $K\cap(y-L)$ are homothetic. By Soltan's characterization of a simplex (\cite{S}), $K=-L$ is a simplex if and only if for every $x\in K+L$ $K\cap x-L$ is homothetic to $K\cap (-L)$. Thus, $K$ and $-L$ are homothetic simplices. Since $K+_{h,1}L=\{0\}$, $K=-L$.
\end{proof}

The following proposition gives an upper inclusion for the $h,\theta$-convolution bodies.
\begin{proposition}\label{UpperInclusion}
Let $K,L\in\K^n_0$ and $h$ like in Definition \ref{Definition} such that for any $v\in S^{n-1}$ $h(K\cap(tv-L))$ is differentiable in an interval $[0,\epsilon)$. Then, for any $\theta\in [0,1)$
$$
\frac{K+_{h,\theta} L}{1-\theta^\frac{1}{k}}\subseteq L_h(K,L),
$$
where
$$L_{h}(K,L):=\left\{x\in\R^n\,:\,-|x|\frac{d^+}{dt}\left.h\left(K\cap \left(t\frac{x}{|x|}-L\right)\right)\right|_{t=0}\leq kM_h(K,L)\right\}.$$
\end{proposition}

\begin{proof}
The concavity of the function $x\to h(K\cap(x-L))^\frac{1}{k}$ implies
\begin{eqnarray*}
h(K\cap(\lambda x-L))&\geq&\left((1-\lambda)M_h(K,L)^\frac{1}{k}+\lambda h(K\cap(x-L))^\frac{1}{k} \right)^k\cr
&=&M_h(K,L)\left[1+\lambda\left(\frac{h(K\cap(x-L))^\frac{1}{k}}{M_h(K,L)^\frac{1}{k}}-1 \right) \right]^k\cr
&\geq&M_h(K,L)\left[1+\lambda k\left(\frac{h(K\cap(x-L))^\frac{1}{k}}{M_h(K,L)^\frac{1}{k}}-1 \right) \right]
\end{eqnarray*}
for $\lambda\in [0,1]$ and $x\in K+L$. On the other hand,
\begin{eqnarray*}
h(K\cap (\lambda x-L))&=&M_h(K,L)+\int_0^{\lambda|x|}\frac{d^+}{dt}h\left(K\cap \left(t\frac{x}{|x|}-L\right)\right)dt\\
&\leq&M_h(K,L)+\lambda|x|\max_{t\in[0,\lambda|x|]}\frac{d^+}{dt}h\left(K\cap \left(t\frac{x}{|x|}-L\right)\right)
\end{eqnarray*}
again using the concavity of $x\to h(K\cap(x-L))^\frac{1}{k}$. Comparing these two inequalities, and letting $\lambda\to0^+$, we obtain
$$
kM_h(K,L)\left(\frac{h(K\cap(x-L))^\frac{1}{k}}{M_h(K,L))^\frac{1}{k}}-1 \right)\leq|x|\frac{d^+}{dt}h\left.\left(K\cap \left(t\frac{x}{|x|}-L\right)\right)\right|_{t=0}.
$$
Since the lateral derivative is non positive, we get the desired inclusion.
\end{proof}

The following lemmas show that, when $K=-L$ is a simplex, all the inclusions above are identities. The first lemma shows that when $K=-L$ is a simplex, then the $h,\theta$-convolution of a linear image of the body is the linear image of the $h,\theta$ convolution.

\begin{lemma}
Let $K$ be a simplex. Then, for any $T\in GL(n)$
$$
TK+_{h,\theta}(-TK)=T(K+_{h,\theta}(-K)).
$$
\end{lemma}

\begin{proof}
By Soltan's result \cite{S}, $K$ is a simplex if and only if for every $x\in K-K$ $K\cap x+K$ is homothetic to $K$. Thus, if $K$ is a simplex, for every $x\in K-K$
$$
K\cap(x+K)=a(x)+\lambda(x)K.
$$
Consequently
\begin{eqnarray*}
K+_{h,\theta}(-K)&=&\{x\in K-K\,:\,h(\lambda(x)K)\geq\theta h(K)\}\cr
&=&\{x\in K-K\,:\,\lambda(x)^k\geq\theta\}.
\end{eqnarray*}

For any $T\in GL(n)$ we have
\begin{eqnarray*}
TK+_{h,\theta}(-TK)&=&\{x\in TK-TK\,:\,h(TK\cap (x+TK))\geq\theta h(TK)\}\cr
&=&\{x\in T(K-K)\,:\,h(T(K\cap (T^{-1}x+K))\geq\theta h(TK)\}\cr
&=&\{x\in T(K-K)\,:\,h(T\lambda(T^{-1}x)K)\geq\theta h(TK)\}\cr
&=&\{x\in T(K-K)\,:\,\lambda(T^{-1}x)^k\geq\theta\}\cr
&=&T(K+_{h,\theta}(-K)).
\end{eqnarray*}
\end{proof}

\begin{lemma}
Let $K\subseteq\R^n$ be a simplex. Then, for any $\theta\in [0,1]$
$$
K+_{h,\theta}(-K)=(1-\theta^{\frac{1}{k}})(K-K).
$$
\end{lemma}

\begin{proof}
The $\supseteq$ part of the identity is a consequence of Corollary \ref{increasingintheta}. By the previous lemma we can assume, without loss of generality, that $K=\textrm{conv}\{0,e_1,\dots,e_n\}$. Then, as it was shown in \cite{AJV},
$$
K\cap(x+K)=a(x)+\lambda(x)K,
$$
with
$$
\lambda(x)=\frac{1}{2}\left(2-\left|\sum_{i=1}^n x_i\right|-\sum_{i=1}^n|x_i|\right).
$$
Consequently,
\begin{eqnarray*}
K+_{h,\theta}(-K)&=&\left\{x\in K-K\,:\,\left|\sum_{i=1}^nx_i\right|+\sum_{i=1}^n|x_i|\leq 2(1-\theta^\frac{1}{k})\right\}\cr
&=&\left(1-\theta^\frac{1}{k}\right)\left\{x\in K-K\,:\,\left|\sum_{i=1}^nx_i\right|+\sum_{i=1}^n|x_i|\leq 2\right\}\cr
&=&\left(1-\theta^\frac{1}{k}\right)(K+_{h,0}(-K))\cr
&=&\left(1-\theta^\frac{1}{k}\right)(K-K).
\end{eqnarray*}
\end{proof}

\begin{lemma}
Let $K\subseteq\R^n$ be a simplex. Then, the set $L_h(K,-K)$ defined in Proposition \ref{UpperInclusion} is
$$
L_h(K,-K)=K-K.
$$
\end{lemma}
\begin{proof}
We can assume, without loss of generality, that $K=\textrm{conv}\{0,e_1,\dots e_n\}$. Then for any $v\in S^{n-1}$
$$
h(K\cap(tv+K))=h(\lambda(tv)K)=\lambda^k(tv)h(K).
$$
with
$$
\lambda(tv)=1-\frac{|t|}{2}\left(\left|\sum_{i=1}^nv_i\right|+\sum_{i=1}^n|v_i|\right).
$$
Consequently
\begin{eqnarray*}
\left.\frac{d}{dt^+}h(K\cap(tv+K))\right|_{t=0}&=&\left.-kh(K)\lambda^{k-1}(tv)\frac{1}{2}\left(\left|\sum_{i=1}^nv_i\right|+\sum_{i=1}^n|v_i|\right)\right|_{t=0}\cr
&=&-kh(K)\frac{1}{2}\left(\left|\sum_{i=1}^nv_i\right|+\sum_{i=1}^n|v_i|\right).
\end{eqnarray*}
Thus
$$
L_h(K,-K)=\left\{x\in\R^n\,:\,\left|\sum_{i=1}^nx_i\right|+\sum_{i=1}^n|x_i|\leq 2\right\}=K-K.
$$
\end{proof}

\section{Lower bound for the volume of the $i-th$ polar projection body}\label{LutwakLowerBound}
In this section we are going to show how inequality (\ref{iZhang}) is deduced from the results in \cite{L3}, and the relation between this inequality and the inequality in Theorem \ref{Theorem}. In \cite{L3}, the author studied the volume of mixed bodies. A particular case of these bodies is the body $[K]_i$ defined by
$$
dS_{n-1}([K]_i,\theta)=dS_{n-i-1}(K,\theta).
$$
The following estimate for their volume was given:
$$
|[K]_i|^{n-1}\leq\frac{W_i(K)^n}{|K|},
$$
with equality if and only if $[K]_i$ and $K$ are homothetic. This reduces to the fact that $K$ is an $(n-i-1)$ tangential body of $B_2^n$ {\it i.e.}, a body such that every support hyperplane of $K$ that is not a support hyperplane of $B_2^n$ contains only $(n-i-2)$ singular points of $K$.

On the other hand, from the definition of $[K]_i$
$$
\Pi^*([K]_{n-i-1})=\Pi_{i}^*(K).
$$
Thus, using Zhang's inequality we obtain
$$
|K|^i|\Pi_i^*(K)|\geq\frac{|K|^i}{|[K]_{n-i-1}|^{n-1}}\frac{1}{n^n}{2n\choose n}\geq\frac{1}{n^n}{2n\choose n}\frac{|K|^{i+1}}{W_{n-i-1}(K)^n}.
$$
There is equality in the above inequalities if and only if $K$ is an $i$-tangential body of a ball and $[K]_{n-i-1}$, which has to be homothetic to $K$, is a simplex. Since the simplex is a $p$-tangential body of $B_2^n$ only for $p=n-1$ there is no equality unless $i=n-1$.

Let $L_{k}(K)=L_{W_{n-k}}(K,-K)$.
The following result shows that the inequality given in Theorem \ref{Theorem} improves inequality (\ref{iZhang}):
\begin{proposition}
Let $K\in\K^n_0$. Then
$$
C_k(K,-K)\subseteq L_k(K)\subseteq nW_{n-k}(K)\Pi_{k-1}^*(K).
$$
\end{proposition}
\begin{proof}
The first inclusion has been shown in Section \ref{ConvolutionBodies}. For the second one, let $v\in S^{n-1}$. Then
\begin{eqnarray*}
&&\frac{d^+}{dt}\left.W_{n-k}\left(K\cap \left(tv+K\right)\right)\right|_{t=0}=\cr
&=&\frac{|B_2^n|}{|B_2^{k}|}\lim_{t\to0^+}\int_{G_{n,k}}\frac{|P_{E}(K\cap(tv+K))|-|P_{E}(K)|}{t}d\nu_{n,k}(E)\cr
&=&\frac{|B_2^n|}{|B_2^{k}|}\int_{G_{n,k}}\lim_{t\to0^+}\frac{|P_{E}(K\cap(tv+K))|-|P_{E}(K)|}{t}d\nu_{n,k}(E)\cr
&\leq&\frac{|B_2^n|}{|B_2^{k}|}\int_{G_{n,k}}\lim_{t\to0^+}\frac{|P_{E}(K)\cap(tP_Ev+P_E(K))|-|P_{E}(K)|}{t}d\nu_{n,k}(E)\cr
&=&-\frac{|B_2^n|}{|B_2^{k}|}\int_{G_{n,k}}|P_{E}v||P_{(P_{E}v)^\perp\cap E}(K)|d\nu_{n,k}(E).
\end{eqnarray*}

For any $k$-dimensional subspace $E$, if $u_1,\dots,u_{n-k}$ is an orthonormal basis of $E^\perp$, we have that

\begin{eqnarray*}|P_{E}v|&=&\sqrt{1-\sum_{i=1}^{n-k}\langle v,u_i\rangle^2}\cr
&=&\sqrt{1-\sum_{i=1}^{n-k}|P_{\textrm{span}\{u_1,\dots,u_{i-1}\}^\perp}v|^2\langle\frac{P_{\textrm{span}\{u_1,\dots,u_{i-1}\}^\perp}v}{|P_{\textrm{span}\{u_1,\dots,u_{i-1}\}^\perp}v|},u_i\rangle^2}
\end{eqnarray*}
and
\begin{eqnarray*} (P_{E}v)^\perp\cap E&=&\textrm{span}\{v,u_1,\dots,u_{n-k}\}^\perp=\textrm{span}\{v,\xi_1,\dots,\xi_{n-k}\}^\perp,
\end{eqnarray*}
where $\xi_1=P_{v^\perp}u_1$ and $\xi_i=P_{\textrm{span}\{v,\xi_1,\dots,\xi_{i-1}\}^\perp} u_i$ ($i>1$).

By uniqueness of the Haar probability measure on $G_{n,k}$, the above integral equals
$$
-\frac{|B_2^n|}{|B_2^{k}|}\int\int\dots\int g_v(u_1,\dots,u_{n-k})d\sigma(u_{n-k})\dots d\sigma(u_1),
$$
where $u_1$ runs over $S^{n-1}$, $u_i$ runs over $S^{n-1}\cap\textrm{span}\{u_1,\dots,u_{i-1}\}^\perp$ ($i>1)$ and
\begin{eqnarray*}
g_v(u_1,\dots,u_{n-k})&=&\sqrt{1-\sum_{i=1}^n|P_{\textrm{span}\{u_1,\dots,u_{i-1}\}^\perp}v|^2\langle\frac{P_{\textrm{span}\{u_1,\dots,u_{i-1}\}^\perp}v}{|P_{\textrm{span}\{u_1,\dots,u_{i-1}\}^\perp}v|},u_i\rangle^2}\times\cr
&\times&|P_{\textrm{span}\{\xi_1,\dots,\xi_{n-k}\}^\perp}P_{v^\perp} (K)|.
\end{eqnarray*}

Now, using the slice integration formula on each one of the spheres, in the direction $\frac{P_{\textrm{span}\{u_1,\dots,u_{i_1}\}^\perp}v}{|P_{\textrm{span}\{u_1,\dots,u_{i_1}\}^\perp}v|}$, we obtain that the previous integral equals
\begin{eqnarray*}
&-&\frac{k}{n}\int_{-1}^1\dots\int_{-1}^1(1-x_1^2)^\frac{n-2}{2}(1-x_2^2)^\frac{n-3}{2}\dots(1-x_{n-k}^2)^\frac{k-1}{2}dx_{n-k}\dots dx_1\times\cr
&\times&\int\int\dots\int|P_{\textrm{span}\{\xi_1,\dots,\xi_{n-k}\}^\perp}P_v^\perp(K)|d\sigma(\xi_{n-k})\dots d\sigma(\xi_1),\cr
\end{eqnarray*}
where $\xi_1$ runs over $S^{n-1}\cap v^\perp$ and $\xi_i$ runs over $S^{n-1}\cap\textrm{span}\{v,\xi_1,\dots,\xi_{i-1}\}^\perp$. By uniqueness of the Haar measure in $G_{v^\perp, k-1}$ equals
\begin{eqnarray*}
&-&\frac{k}{n}\int_{-1}^1\dots\int_{-1}^1(1-x_1^2)^\frac{n-2}{2}(1-x_2^2)^\frac{n-3}{2}\dots(1-x_{n-k}^2)^\frac{k-1}{2}dx_{n-k}\dots dx_1\times\cr
&\times&\int_{G_{v^\perp,k-1}}|P_E P_{v^\perp}(K)|d\nu_{n-1,k-1}\cr
&=&-\frac{k|B_2^{k-1}|}{n|B_2^{n-1}|}\int_{-1}^1\dots\int_{-1}^1(1-x_1^2)^\frac{n-2}{2}(1-x_2^2)^\frac{n-3}{2}\dots(1-x_{n-k}^2)^\frac{k-1}{2}dx_{n-k}\dots dx_1\times\cr
&\times& W_{n-k}(P_{v^\perp}(K))\cr
&=&-\frac{k|B_2^{k-1}|}{n|B_2^{n-1}|}\frac{(\sqrt{\pi})^{n-k}\Gamma\left(\frac{k+1}{2}\right)}{\Gamma\left(\frac{n+1}{2}\right)}W_{n-k}(P_{v^\perp}(K))=-\frac{k}{n}\Vert v\Vert_{\Pi_{k-1}^*(K)}.
\end{eqnarray*}
Consequently
$$
L_k(K)\subseteq nW_{n-k}(K)\Pi_{k-1}^*(K).
$$
\end{proof}



\section{Rogers-Sephard inequality and Zhang's inequality for $C_{n-1}(K)$}\label{Proof}

In this section we prove Theorem \ref{Theorem}. It is a consequence of Theorem \ref{BoundVolumeLimiting} and the following:
\begin{thm}
Let $K\in\K^n_0$, $h$  a function like in Definition \ref{Definition} and
$$
C_h(K,L):=\lim_{\theta\to1^-}\frac{K+_{h,\theta}L}{1-\theta^{\frac{1}{k}}}
$$
Then
$$
|C_h(K,L)|\geq {n+k \choose n}\int_{\R^n} \frac{h(K\cap (x-L))}{M_h(K,L)}dx\geq |K+L|,
$$
with equality when $K=-L$ is a simplex. If $h$ is like in Lemma \ref{simplex}, then there is equality if and only if $K=-L$ is a simplex.
\end{thm}
\begin{proof}
By Proposition \ref{UpperInclusion}, for any $\theta\in [0,1)$
$$|C_h(K,L)|(1-\theta^\frac{1}{k})^n\geq |K+_{h,\theta} L|\geq|K+L|(1-\theta^\frac{1}{k})^n.$$
Thus
$$|C_h(K,L)|\int_0^1(1-\theta^\frac{1}{k})^nd\theta\geq \int_0^1|K+_{h,\theta} L|d\theta\geq|K+L|\int_0^1(1-\theta^\frac{1}{k})^nd\theta.$$

Since
\begin{eqnarray*}
\int_0^1|K+_{h,\theta} L|d\theta&=&\int_0^1\int_{\R^n}\chi_{h(K\cap(y-L))\geq\theta M_h(K,L)}(x)dxd\theta\cr
&=&\int_{\R^n} \frac{h(K\cap (x-L))}{M_h(K,L)}dx
\end{eqnarray*}
we obtain the result. By the Lemmas in the previous Section, all the inequalities are equalities when $K=-L$ is a simplex and if $h$ is like in Lemma \ref{simplex}, then there is equality if and only if $K=-L$ is a simplex.
\end{proof}

Taking $h(K)=W_{n-k}(K)$, we obtain the following Theorem, which in particular gives Theorem \ref{Theorem}, since the inequality we obtain computing the integral $\int_{\R^n} \frac{h(K\cap (x-L))}{M_h(K,L)}dx$ is an equality when $h(K)=W_1(K)$:
\begin{thm}\label{BoundVolumeLimiting}
Let $K\in\K^n_0$. Then, for any $1\leq k\leq n$
$$
|C_k(K,L)|\geq {n+k \choose n}\frac{|K|W_{n-k}(L)+|L|W_{n-k}(K)}{W_{n-k}(K\cap (-L))}.
$$
If $L=-K$ we can slightly improve this to
$$
|C_k(K,-K)|\geq{2n \choose n}{2n\choose n-k}^{-1}\left(2{n\choose k}+2^{n-k}-2\right)|K|.
$$
When $k=n-1$ these inequalities are sharp and we have equality if and only if $K=-L$ is a simplex.
\end{thm}
\begin{proof}
If we take $h(K)=W_{n-k}(K)$ we have, by Crofton's intersection formula (see \cite{SCH}, page 235) that
$$
W_{n-k}(K)=C_{n,k}\mu_{n,n-k}\{E\in\mathbb A_{n,n-k}\,:\,K\cap E\neq\emptyset\},
$$
where $C_{n,k}$ is a constant depending only on $n$ and $k$ and $d\mu_{n,n-k}$ is the Haar measure on the set of affine $(n-k)$-dimensional subspaces of $\R^n$, $\mathbb A_{n,n-k}$. Thus
\begin{eqnarray*}
\int_{\R^n} \frac{h(K\cap (x-L))}{M_h(K,L)}dx&=&\frac{\int_{\R^n}\int_{\mathbb A_{n,n-k}}\chi_{\{K\cap(x-L)\cap E\neq\emptyset\}}(E)d\mu_{n,n-k}(E)dx}{\mu_{n,n-k}\{E\in\mathbb A_{n,n-k}\,:\,K\cap(-L)\cap E\neq\emptyset\}}\cr
&=&\frac{\int_{\{E\in\mathbb A_{n,n-k}\,:\,K\cap E\neq\emptyset\}}|(K\cap E)+L|d\mu_{n,n-k}(E)}{\mu_{n,n-k}\{E\in\mathbb A_{n,n-k}\,:\,K\cap (-L)\cap E\neq\emptyset\}}\cr
\end{eqnarray*}
For every $E\in\mathbb A_{n,n-k}$, calling $E_0$ the linear subspace parallel to $E$,
$$
|(K\cap E)+L|=\int_{P_{E_0^\perp}L}|(K\cap E)+(L\cap(y+E_0))|dy.
$$
Thus, since for any subspace $E_0\in G_{n,k}$, ${n\choose k}\max_{x\in E_0^\perp}|K\cap(x+E_o)||P_{E_0^\perp}(K)|\leq|K|$ (see \cite{Pi}, Lemma 8.8 for a proof in the symmetric case, which also works in the non-symmetric case).
\begin{eqnarray*}
&&\int_{\{E\in\mathbb A_{n,n-k}\,:\,K\cap E\neq\emptyset\}}|(K\cap E)+L|d\mu_{n,n-k}(E)\cr&=&\int_{G_{n,n-k}}\int_{P_{E_0^\perp}(K)}\int_{P_{E_0^\perp}(L)}|(K\cap (z+E_0))+(L\cap(y+E_0))|dydzd\nu_{n,n-k}(E_0)\cr
&\geq&\int_{G_{n,n-k}}\int_{P_{E_0^\perp}(K)}\int_{P_{E_0^\perp}(L)}\left(|(K\cap (z+E_0))|^\frac{1}{n-k}+|(L\cap(y+E_0))|^\frac{1}{n-k}\right)^{n-k}\times\cr
&\times&dydzd\nu_{n,n-k}(E_0)\cr
&\geq&|K|\int_{G_{n,n-k}}|P_{E_0^\perp}(L)|d\nu_{n,n-k}+|L|\int_{G_{n,n-k}}|P_{E_0^\perp}(K)|d\nu_{n,n-k},\cr
\end{eqnarray*}
where the first inequality follows from the $(n-k)$-dimensional version of Brunn-Minkowski inequality and the second one follows
from the fact that $(a+b)^{n-k}\geq a^{n-k}+b^{n-k}$ for any $a,b\geq 0$.

Since
\begin{eqnarray*}
\mu_{n,n-k}\{E\in\mathbb A_{n,1}\,:\,K\cap (-L)\cap E\neq\emptyset\}&=&\int_{G_{n,n-k}}|P_{E_0}^\perp(K\cap (-L))|d\nu_{n,n-n}(E_0)\cr
&=&\frac{|B_2^{k}|}{|B_2^n|}W_{n-k}(K\cap (-L))
\end{eqnarray*}
we have
$$
\int_{\R^n} \frac{W_{n-k}(K\cap (x-L))}{W_{n-k}(K\cap(-L))}dx\geq\frac{|K|W_{n-k}(L)+|L|W_{n-k}(K)}{W_{n-k}(K\cap (-L))}
$$
Thus
\begin{eqnarray*}
|C_k(K,L)|&\geq&{n+k \choose n}\frac{|K|W_{n-k}(L)+|L|W_{n-k}(K)}{W_{n-k}(K\cap L)}.
\end{eqnarray*}
Notice that if $k=n-1$ the above inequalities become equalities.
If $L=-K$, we have
\begin{eqnarray*}
&&\int_{\{E\in\mathbb A_{n,n-k}\,:\,K\cap E\neq\emptyset\}}|(K\cap E)-K|d\mu_{n,n-k}(E)\cr&=&\int_{G_{n,n-k}}\int_{P_{E_0^\perp}(K)}\int_{P_{E_0^\perp}(-K)}|(K\cap (z+E_0))+((-K)\cap(y+E_0))|\times\cr
&\times&dydzd\nu_{n,n-k}(E_0)\cr
&\geq&\int_{G_{n,n-k}}\int_{P_{E_0^\perp}(K)}\int_{P_{E_0^\perp}(-K)}\left(|K\cap (z+E_0)|^\frac{1}{n-k}+|(-K)\cap(y+E_0)|^\frac{1}{n-k}\right)^{n-k}\cr
&\times&dydzd\nu_{n,n-k}(E_0)\cr
&\geq&\int_{G_{n,n-k}}\int_{P_{E_0^\perp(K)}}\int_{P_{E_0^\perp}(-K)}\sum_{i=0}^{n-k}{{n-k}\choose i}|K\cap (z+E_0)|^{\frac{i}{n-k}}\times\cr
&\times&|(-K)\cap(y+E_0)|^{\frac{n-k-i}{n-k}}dydzd\nu_{n,n-k}(E_0)\cr
&=&2|K|\int_{G_{n,n-k}}|P_{E_0^\perp}(K)|d\nu_{n,n-k}\cr
&+&\sum_{i=1}^{n-k-1}{n-k\choose i}\int_{G_{n,n-k}}\int_{P_{E_0^\perp(K)}}\int_{P_{E_0^\perp}(-K)}\frac{|K\cap (z+E_0)|}{|K\cap (z+E_0)|^\frac{n-k-i}{n-k}}\times\cr
&\times&\frac{|(-K)\cap(y+E_0)|}{|(-K)\cap(y+E_0)|^\frac{i}{k}}dydzd\nu_{n,n-k}\cr
&\geq&2|K|\int_{G_{n,n-k}}|P_{E_0^\perp}(K)|d\nu_{n,n-k}\cr
&+&\sum_{i=1}^{n-k-1}{n-k\choose i}\int_{G_{n,n-k}}\int_{P_{E_0^\perp(K)}}\int_{P_{E_0^\perp}(-K)}|K\cap (z+E_0)|\times\cr
&\times&\frac{|(-K)\cap (y+E_0)|}{\max_{x\in P_E(K)}|K\cap (x+E_0)|}dydzd\nu_{n,n-k}\cr
&\geq&2|K|\int_{G_{n,n-k}}|P_{E_0^\perp}(K)|d\nu_{n,n-k}\cr
&+&(2^{n-k}-2)\int_{G_{n,n-k}}\frac{|K|^2}{\max_{x\in P_E(K)}|K\cap (x+E_0)|}d\nu_{n,n-k}\cr
&\geq&2|K|\int_{G_{n,n-k}}|P_{E_0^\perp}(K)|d\nu_{n,n-k}\cr
&+&(2^{n-k}-2){n\choose k}^{-1}|K|\int_{G_{n,n-k}}|P_{E_0^\perp(K)}|d\nu_{n,n-k}\cr
&=&\left(2{n\choose k}+2^{n-k}-2\right){n\choose k}^{-1}|K|\frac{|B_2^k|}{|B_2^n|}W_{n-k}(K).
\end{eqnarray*}
and then
$$
\int_{\R^n} \frac{W_{n-k}(K\cap (x+K))}{W_{n-k}(K)}dx\geq\left(2{n\choose k}+2^{n-k}-2\right){n\choose k}^{-1}|K|.
$$
Thus
\begin{eqnarray*}
|C_k(K,-K)|&\geq&{n+k \choose n}{n\choose k}^{-1}\left(2{n\choose k}+2^{n-k}-2\right)|K|\cr
&=&{2n \choose n}{2n\choose n-k}^{-1}\left(2{n\choose k}+2^{n-k}-2\right)|K|.
\end{eqnarray*}
\end{proof}

\section{Sections of the difference body and the polar projection body}\label{SectionsProjDiff}

In the following proposition we use the inclusion relation we obtained for the $h,\theta$- convolution bodies (for $h$ being the volume of the projection onto a subspace) to give an estimate for the volume of the sections of the Minkowski sum of two convex bodies. In particular, taking $h$ the volume (which is the volume the projection onto $\R^n$) we can give a simpler proof of the upper bound in (\ref{SectionsDifferenceBody}) involving the $\frac{n}{k}$ term.

\begin{proposition}
Let $E\in G_{n,k}$ be a linear subspace and let $F\in G_{n,l}$ be a linear subspace such that $E\subseteq F$. Then, for any $K,L$ convex bodies we have
$$
|(K+L)\cap E|\leq{l+k\choose k}\int_{F\cap E^\perp}\frac{|P_F(K)\cap(x+E)||P_F(-L)\cap(x+E)|}{\max_{z\in\R^n}|P_F(K\cap(z-L)|}dx
$$
In particular, if $L=-K$ we obtain the following estimate for the volume of the sections of the difference body
\begin{eqnarray*}
|(K-K)\cap E|&\leq&{l+k\choose k}\inf_{F\in G_{n,l}, E\subseteq F}\max_{x\in F}|P_F(K)\cap(x+E)|
\end{eqnarray*}
\end{proposition}

\begin{proof}
Let $h(K)$=$P_F(K)$. By Corollary \ref{increasingintheta}, we have that
$$
(1-\theta^\frac{1}{l})^k((K+L)\cap E)\subseteq (K+_{h,\theta}L)\cap E.
$$
Thus, taking volumes and integrating in $[0,1]$ we obtain
$$
{k+l\choose k}^{-1}|(K+L)\cap E|\leq\int_0^1|(K+_{h,\theta}L)\cap E|d\theta.
$$
Now, since $E\subseteq F$,
\begin{eqnarray*}
\int_0^1|(K+_{h,\theta}L)\cap E|d\theta&=&\int_E\frac{|P_F(K\cap(x-L))|}{M_h(K,L)}dx\cr
&\leq&\int_E\frac{|P_F(K)\cap(x-P_F(L)))|}{M_h(K,L)}dx\cr
&=&\frac{1}{M_h(K,L)}\int_E\int_F\chi_{P_F(K)}(y)\chi_{x-P_F(L)}(y)dydx\cr
&=&\frac{1}{M_h(K,L)}\int_F\int_E\chi_{P_F(K)}(y)\chi_{y+P_F(L)}(x)dxdy\cr
&=&\frac{1}{M_h(K,L)}\int_F\chi_{P_F(K)}(y)|(y+P_F(L))\cap E|dy\cr
&=&\int_{F\cap E^\perp}\frac{|P_F(K)\cap(z+E)||(-P_F(L))\cap (z+E)|}{M_h(K,L)}dz\cr
\end{eqnarray*}
In particular, if $L=-K$
\begin{eqnarray*}
|(K-K)\cap E|&\leq&{l+k\choose k}\inf_{F\in G_{n,l}, E\subseteq F}\int_{F\cap E^\perp}\frac{|P_F(K)\cap(x+E)|^2}{|P_F(K)|}dx\cr
&\leq&{l+k\choose k}\inf_{F\in G_{n,l}, E\subseteq F}\max_{x\in F}|P_F(K)\cap(x+E)|
\end{eqnarray*}
\end{proof}
\begin{rmk}
If we take $L=-K$, $F=\R^n$, we obtain
\begin{eqnarray*}
|(K-K)\cap E|&\leq&{n+k\choose k}\max_{x\in \R^n}|P_F(K)\cap(x+E)|\cr
&\leq&e^k\left(1+\frac{n}{k}\right)^k\max_{x\in \R^n}|K\cap(x+E)|
\end{eqnarray*}
and recover one of the two upper bounds proved in (\ref{SectionsDifferenceBody}) for the volume of the sections of the difference body.
\end{rmk}

In the same way we can give a lower bound for the volume of the sections of the polar projection body of a convex body:

\begin{proposition}
Let $E\in G_{n,k}$ be a linear subspace. Then, for any $K,L$ convex bodies we have
$$
|C_n(K,L)\cap E|\geq{n+k\choose n}\int_{E^\perp}\frac{|K\cap(x+E)||(-L)\cap(x+E)|}{M_0(K,L)}dx.
$$
When $L=-K$
$$
n^k|K|^k|\Pi^*(K)\cap E|\geq{n+k\choose n}\frac{|K|}{|P_{E^\perp}(K)|}.
$$
\end{proposition}

\begin{proof}
By Corollary \ref{increasingintheta}, we have that
$$
(1-\theta^\frac{1}{n})C_n(K,L)\cap E\supseteq (K+_{n,\theta}L)\cap E.
$$
Taking volumes and integrating in $[0,1]$ we have
$$
{n+k\choose n}^{-1}|C_n(K,L)\cap E|\geq\int_0^1|(K+_{n,\theta}L)\cap E|d\theta.
$$
Now,
\begin{eqnarray*}
\int_0^1|(K+_{n,\theta}L)\cap E|d\theta&=&\int_E\int_0^1\chi_{\{x\in\R^n\,:\,|K\cap(x-L)|\geq\theta M_0(K,L)\}}(z)d\theta dz\cr
&=&\int_E\frac{|K\cap(z-L)|}{M_0(K,L)}dz=\frac{\int_E\int_{\R^n}\chi_K(y)\chi_{z-L}(y)dydz}{M_0(K,L)}\cr
&=&\frac{\int_E\int_{\R^n}\chi_K(y)\chi_{y+L}(z)dydz}{M_0(K,L)}\cr
&=&\frac{\int_{\R^n}\chi_K(y)|(y+L)\cap E|dy}{M_0(K,L)}\cr
&=&\frac{\int_{\R^n}\chi_K(y)|(-L)\cap (y+E)|dy}{M_0(K,L)}\cr
&=&\int_{E^\perp}\frac{|K\cap (x+E)||(-L)\cap (x+E)|dx}{M_0(K,L)}.\cr
\end{eqnarray*}

In particular, if $L=-K$, this integral equals
\begin{eqnarray*}
\frac{1}{|K|}\int_{E^\perp}|K\cap(x+E)|^2dx&=&\frac{|P_{E^\perp}(K)|}{|K|}\frac{1}{|P_{E^\perp}(K)|}\int_{E^\perp}|K\cap(x+E)|^2dx\cr
&\geq&\frac{|P_{E^\perp}(K)|}{|K|}\left(\frac{1}{|P_{E^\perp}(K)|}\int_{E^\perp}|K\cap(x+E)|dx\right)^2\cr
&=&\frac{|K|}{|P_{E^\perp}(K)|}.
\end{eqnarray*}
\end{proof}

\section{Acknowledgements}
D. Alonso Guti\'errez was partially supported by MICINN project MTM2010-16679,
MICINN-FEDER project MTM2009-10418, ``Programa de Ayudas a Grupos de
Excelencia de la Regi\'on de Murcia'', Fundaci\'on S\'eneca,
04540/GERM/06 and Institut Universitari de Matem\` atiques i Aplicacions de Castell\'o.

B. Gonz\'alez was partially supported by MINECO (Ministerio de Econom\'ia y Competitividad) and FEDER (Fondo Europeo de Desarrollo Regional) project MTM2012-34037, and Fundaci\'{o}n S\'{e}neca project 04540/GERM/06, Spain. This research is a result of the activity developed within the framework of the Programme in Support of Excellence Groups of the Regi\'{o}n de Murcia, Spain, by Fundaci\'{o}n S\'{e}neca, Regional Agency for
Science and Technology (Regional Plan for Science and Technology 2007-2010).

C. Hugo Jim\'enez was partially supported by the Spanish Ministry of Economy and Competitiveness, grant MTM2012-30748 and by Mexico's National Council for Sciences and Technology (CONACyT) postdoctoral grant 180486.

\end{document}